\documentclass{amsart}

\usepackage{amssymb}
\usepackage{color}
\usepackage{graphicx}

\newcommand{\rmd}{\mathrm{d}}

\newcommand{\Jac}{\mathrm{Jac}}

\DeclareMathOperator{\Complex}{\mathbb{C}}
\DeclareMathOperator{\Integer}{\mathbb{Z}}

\DeclareMathOperator{\wgt}{wgt}

\newtheorem{theo}{Theorem}
\newtheorem{prop}{Proposition}
\newtheorem{cor}{Corollary}

\theoremstyle{definition}

\newtheorem{rem}{Remark}

\theoremstyle{plain}

\title[]{Division polynomials in Mumford coordinates}
\author{J Bernatska}
\address{}
\email{}
\date{\today}

\allowdisplaybreaks[4]

\begin{document}
 
\maketitle 
\begin{abstract}
An effective method of computing division polynomials in terms of  Mumford coordinates is presented.
As an example, division polynomials for $3$- and $4$-torsion divisors on a genus two curve
are obtained explicitly in terms of  Mumford coordinates, and $x$-, $y$-coordinates of the support of
torsion divisors.
As a result, $n$-torsion divisors on a given curve can be computed directly from the division polynomials.
Alternatively, these divisors are obtained by solving the Jacobi inversion problem at points of the Jacobian variety
of order $n$.
\end{abstract}

\section{Introduction}
The research projects like Crypto-Math CREST \cite{MMNGC2018}, and the stream of publications in 
IACR Cryptology ePrint Archive show that the
isogeny-based cryptography is considered as a part of the
next-generation cryptography.
And genus two curves  have essential benefits in public-key cryptography, as proven, for example, in
\cite{BL2014,BCHL2013}

Below, we focus on division polynomials, which arose in the elliptic case to define
$n$-torsion points on the curve. Generalizations to higher genera are suggested in 
\cite{Cant1994,Kanayama,Uchida}.

In \cite{Cant1994},  polynomials which define the reduced representation of  
divisors $n D$,  $D=(x,y)-\infty$,
are constructed, and called division polynomials. 
(By $n D$ the sum $D+D+\cdots +D$ with $n$ terms is denoted, and we adopt this notation.)
On the contrary, 
computations show that $n$-torsion divisors on a hyperelliptic curve are non-special,
except the case of $2$-torsion divisors,
and no divisors contain repeated points. Thus, divisors of the form $k \cdot (x,y)$ are not $n$-torsion for any $n$.

In \cite{Kanayama}, we find a division polynomial in the form $\phi_n(u) = \sigma(n u)/ \sigma(u)^{n^2}$,
$u\in \Jac(\mathcal{C})\backslash (\sigma)_0$, 
where $(\sigma)_0 = \{u\in \Jac(\mathcal{C}) \mid \sigma(u)=0\}$
denotes the theta divisor.
Polynomials $\phi_n$ are expressible in terms of fixed set of $\wp$-functions.
A curve $\mathcal{C}$ of genus two  is considered. 
The expression for $\phi_2$ is adopted from \cite[p.\,100]{bakerMP}. 
For $n\geqslant 3$ a recurrence relation is suggested,
as well as a method of constructing $\phi_n$. According to \cite[Theorem\,7]{Kanayama},
$n$-torion points on $\Jac(\mathcal{C})[n] \backslash ((\sigma)_0 \cap \Jac(\mathcal{C})[n])$
are the common zeros of the equations of $\Jac(\mathcal{C})$,  
$\phi_n$, $\partial_{u_1}\phi_n$, $\partial^2_{u_1}\phi_n$, where $\partial_u \equiv \partial/\partial u$.

In \cite{Uchida}, Kanayama's division polynomials are generalized to 
the case of a general hyperelliptic Jacobian, though all examples do not exceed genus two, and 
only the case of $n=2$ in genus two is considered in detail.
A more general case is proven for special divisors composed from one point only,
similar to the determinant formulas adopted from \cite{Onishi}.

The method suggested in \cite{Kanayama} and developed in \cite{Uchida}
requires essential efforts, and the knowledge of identities for 
$\wp$-functions associated with the curve in question.
Moreover, obtaining  multiplication
formulas of the form $\sigma(n u)/ \sigma(u)^{n^2}$ is a complicate problem itself.
Below, a simpler method of obtaining division polynomials is suggested. 
Under division polynomials we will understand the polynomials which define $n$-torsion divisors
on a curve $\mathcal{C}$.
The method leads to purely algebraic computations, though it is based on the structure of the field
of abelian $\wp$-functions associated with $\mathcal{C}$. 
As a result, $x$-, $y$-coordinates of the support of $n$-torsion divisors
are obtained.

There is no problem to find $n$-torsion points on $\Jac(\mathcal{C})$  ---
these are points of order~$n$. The problem arises when $x$-, $y$-coordinates  of 
the Abel pre-images of such points are required. 
There exists a direct way to compute the non-special Abel pre-image from a given
point of $\Jac(\mathcal{C})$, based on a solution of the Jacobi inversion problem.
Such a solution is known in terms of  $\wp$-functions, see
\cite[\S\,216]{bakerAF} for hyperelliptic curves, and \cite{BLJIP22} for non-hyperelliptic curves.
This way became feasible due to the progress 
in computing uniformization of plane algebraic curves,
see \cite{BerCalc24}.
In the computations presented below, we use this way for verification.

The proposed method of obtaining division polynomials 
is based on the addition and duplication laws,
written in terms of the Mumford coordinates of $n$-torsion divisors.
As will be shown, defining  $n$-torsion divisors on a genus $g$ curve
requires not more than $g$ polynomials. The number of polynomials
 decreases in the case of even $n$, say $n = 2k$, when $k D$ is a special divisor. 
Note, that $k D$ produced from an $n$-torsion divisor $D$ is $2$-torsion, if $n=2k$.

For the sake of compact expressions, we work with a genus two curve $\mathcal{C}$. 
Addition on the Jacobian variety $\Jac(\mathcal{C})$ of such a curve is widely known, see \cite{Cant1987,CL2011}.
This approach arises from Cantor's algorithm, and uses
 Mumford's representation of divisors on a hyperelliptic curve $f(x,y)=0$
by interpolation polynomials in $x$-coordinate. 

Below, the addition law is adopted from \cite{bl2005}. 
It  is based on the theory of polynomial functions on $\mathcal{C}$, 
which form a ring $\mathfrak{P}(\mathcal{C}) = \Complex[x,y] / f(x,y;\lambda)$.
The structure of the ring
 is closely connected to the field of $\wp$-functions associated with $\mathcal{C}$.
Polynomial  functions are composed of  monomials
in $x$, $y$ arranged by the Sato weight. The theory of polynomial functions on $\mathcal{C}$
contains an elegant technique 
which replaces the mentioned interpolation polynomials, and makes the addition law
explicit, and easy to derive.

The paper is organized as follows.
In Section 2  basic notions are recalled and notations are introduced.
In Section 3 the ring $\mathfrak{P}(\mathcal{C})$ of polynomial functions 
on a hyperelliptic curve $\mathcal{C}$ is described in detail;
 these functions are used to define  divisors, and implement
addition and inversion on the curve. In Section~4
the addition and duplication laws are derived by means 
of polynomial functions from $\mathfrak{P}(\mathcal{C})$.
The cases of special divisors are also addressed.
Section~5 is devoted to $n$-torsion divisors, and derivation of division polynomials in the Mumford coordinates,
and $x$-, $y$-coordinates.

The method is illustrated by obtaining the division polynomials for $3$- and $4$-torsion divisors.
Computations of these division polynomials on a given curve,
as well as the corresponding torsion divisors, are implemented in Wolfram Mathematica 12, 
see \texttt{https://community.wolfram.com/groups/-/m/t/3338527}.

The proposed method can be easily extended to a hyperelliptic curve of  any genus,
as well as to non-hyperelliptic curves.

\section{Preliminaries}

\subsection{Hyperelliptic curve of genus two}
Let a genus two curve $\mathcal{C}$ be defined by the equation
\begin{align}\label{C25Eq}
0 = f(x,y;\lambda) &= -y^2  + \mathcal{P}(x) \\
&= -y^2  + x^5 + \lambda_2 x^4 + \lambda_4 x^3 + \lambda_6 x^2 + \lambda_8 x + \lambda_{10}, \notag
\end{align}
which we call the canonical form, unlike \cite{PMACG1996}, but 
according to the theory of multi-variable $\sigma$-functions \cite{bel2012}.

The canonical form of a plane algebraic curve $\mathcal{C}$, also known as $(n,s)$-curve, 
$\gcd(n,s)=1$, see \cite{bel99},
is the Weierstass canonical form described in \cite[\S\S 60–63]{bakerAF}.
Every $(n,s)$-curve is equipped with 
a unique modular-invariant, entire  $\sigma$-function, which  
 has a representation as an analytic series in $u\in \Jac(\mathcal{C})$
and parameters $\lambda$ of~$\mathcal{C}$, see \cite[Ch.\;9]{bel2012}.
The differential field of abelian functions generated from $\sigma$-function, which we call $\wp$-functions,
gives rise to an algebraic model of $\Jac(\mathcal{C})$, see \cite[Ch.\,3]{bel2012}, and 
the addition law on $\mathcal{C}$, see \cite{bl2005}.

We work over the field $\Complex$ of complex numbers, $(x,y) \in \Complex^2$;
and assume that the curve $\mathcal{C}$ defined by \eqref{C25Eq} is not degenerate, 
that is $\lambda \in \Complex^5 \backslash 
\mathrm{Dscr} $, where $\mathrm{Dscr}$ consists of such $\lambda$ that 
the genus of  $\mathcal{C}$ reduces to $1$ or $0$. 
In more detail strata of the space 
 of parameters $\lambda$ are described in \cite{BerLey2019}.

The theory of $\sigma$- and $\wp$-functions associated with $(n,s)$-curves respects the Sato weight, 
which shows the order of zero at infinity\footnote{
The infinity point on an $(n,s)$-curve is a Weierstrass point, 
and a branch point where all $n$ sheets wind.}.
Due to the leading terms $-y^2  + x^5$ with co-prime exponents,
the curve \eqref{C25Eq} admits the following expansion about infinity in a local parameter~$\xi$:
\begin{gather}\label{param}
 \begin{split}
 x = \xi^{-2},\quad y &=
 \xi^{-5} \big(1+ \tfrac{1}{2} \lambda_2 \xi^2
+ \tfrac{1}{2} (\lambda_4 - \tfrac{1}{4} \lambda_2^2)\xi^4 
+ \tfrac{1}{2} (\lambda_6 - \tfrac{1}{2} \lambda_2 \lambda_4 + \tfrac{1}{8} \lambda_2^3)\xi^6 \\
&\qquad  + \tfrac{1}{2} (\lambda_8 - \tfrac{1}{2} \lambda_2 \lambda_6 - \tfrac{1}{2} \lambda_4^2
+ \tfrac{3}{8} \lambda_2^2 \lambda_4 - \tfrac{5}{64} \lambda_2^4)\xi^8 \\
&\qquad
+ \tfrac{1}{2} (\lambda_{10} - \tfrac{1}{2} \lambda_2 \lambda_8 - \tfrac{1}{2} \lambda_4 \lambda_6
+ \tfrac{3}{8} \lambda_2^2 \lambda_6 + \tfrac{3}{8} \lambda_2 \lambda_4^2 \\
&\qquad - \tfrac{5}{16} \lambda_2^3 \lambda_4 + \tfrac{7}{128} \lambda_2^5)\xi^{10}
+ O(\xi^{12}) \big).
\end{split}
\end{gather}
The negative exponents of leading terms show the Sato weights:
$\wgt x = 2$, $\wgt y\,{=}\, 5$. The  weight introduces an order in the list of monomials in $x$ and $y$:
\begin{equation}\label{MlListC25}
\begin{array}{lccccccccc}
\text{weights:} & 0 & 2 & 4 & 5 & 6 & 7 & 8 & 9 &  \\
\mathfrak{M} = \{& 1, &x, &x^2, & y, & x^3, & y x, & x^4 & y x^2, & \dots\}
\end{array}
\end{equation}
The absent weights $\{1,\,3\}$ form the Weierstras  gap sequence of $\mathcal{C}$.

Let holomorphic (or first kind) differentials $\rmd u = (\rmd u_1,\rmd u_3)^t$ be given 
in the standard not normalized  form:
\begin{gather}\label{FKDiffDef}
\rmd u_1 = \frac{x \rmd x}{-2y}, \qquad\qquad \rmd u_3 = \frac{\rmd x}{-2y}. 
\end{gather}
Note, that $\wgt \rmd u_{\mathfrak{w}} = -\mathfrak{w}$, and the weights coincide with the negative
Weierstrasss gap sequence. 
The Abel map is defined with respect to these differentials, with the basepoint located at infinity, namely
\begin{align*}
&\mathcal{A}(P) = \int_{\infty}^P \rmd u, \ \ \quad \qquad  P = (x,y) \in \mathcal{C};\\
&\mathcal{A}(D) = \sum_{k=1}^n \mathcal{A}(P_k), \qquad D = \sum_{k=1}^n P_k
\in \mathcal{C}^n.
\end{align*}

Let $\mathfrak{a}_1$, $\mathfrak{b}_1$, $\mathfrak{a}_2$, $\mathfrak{b}_2$
form a canonical homology basis.
Not normalized first kind period matrices are defined by
\begin{gather}\label{FKPeriods}
\omega = (\omega_{i,j}) = \bigg(\int_{\mathfrak{a}_j} \rmd u_i\bigg), \qquad\qquad
\omega' = (\omega'_{i,j}) = \bigg(\int_{\mathfrak{b}_j} \rmd u_i\bigg),
\end{gather}
and generate the period lattice $\{\omega,\omega' \}$, which introduces a polarization
on the Jacobian variety 
$\Jac(\mathcal{C}) = \Complex^2 / \{\omega,\omega' \}$.
Similarly to \eqref{FKPeriods}, second kind  period matrices are defined:
\begin{gather}\label{etaM}
 \eta = (\eta_{ij})= \bigg( \int_{\mathfrak{a}_j} \rmd r_i\bigg),\qquad\quad
 \eta' = (\eta'_{ij}) = \bigg(\int_{\mathfrak{b}_j} \rmd r_i \bigg),
\end{gather}
from second kind differentials $\rmd r = (\rmd r_1, \rmd r_3)^t$
associated  with the first kind differentials $\rmd u$, see \cite[Art.\,138]{bakerAF}. 
Namely,
\begin{gather*}
\rmd r_1 =   \frac{x^2 \rmd x}{-2y},\qquad\quad
\rmd r_3 = \big(3x^3 + 2 \lambda_2 x^2 + \lambda_4 x \big) \frac{x \rmd x}{-2y}.
\end{gather*}

Let $D_2 = (x_1,y_1) + (x_2,y_2)$ be a  non-special divisor on $\mathcal{C}$.
The Abel image $\mathcal{A}(D_2) \equiv u = (u_1,u_3)^t$ is computed by
\begin{gather}\label{FKIntDef}
u = \int_\infty^{(x_1,y_1)} \rmd u + \int_\infty^{(x_2,y_2)} \rmd u.
\end{gather}

The not normalized coordinates $u = (u_1,u_3)^t$ serve as an argument of $\sigma$-, and $\wp$-functions.
Normalization, which is employed by $\theta$-function, is reached as follows
\begin{equation}
v = \omega^{-1} u,\qquad\qquad \tau = \omega^{-1} \omega',
\end{equation}
where $v$ is a vector of  normalized coordinates of $\Jac(\mathcal{C})$, 
and $\tau$ belongs to the Siegel space of order $2$.

\subsection{Entire functions}
Let the Riemann  $\theta$-function on $\Complex^2 \supset \Jac(\mathcal{C})$ be defined by
\begin{gather}\label{ThetaDef}
 \theta(v;\tau) = \sum_{n\in \Integer^2} \exp \big(\imath \pi n^t \tau n + 2\imath \pi n^t v\big).
\end{gather}
Let a $\theta$-function with characteristic $[\varepsilon]= (\varepsilon', \varepsilon)^t$ be defined by
\begin{equation}\label{ThetaDefChar}
 \theta[\varepsilon](v;\tau) = \exp\big(\imath \pi ( \varepsilon'{}^t \tau \varepsilon'
 + 2 \imath \pi  (v+\varepsilon)^t (\varepsilon')\big)  \theta(v+\varepsilon + \tau \varepsilon';\tau),
\end{equation}
where $[\varepsilon]$ is a $2\times g$ matrix 
composed of two $2$-component vectors $\varepsilon'$ and $\varepsilon$ with 
real entries within the interval $[0,1)$.  

According to \cite[Eq.(2.3)]{belHKF}, $\sigma$-function is related to $\theta$-function
 as follows
\begin{equation}\label{SigmaThetaRel}
\sigma(u) = C \exp\big({-}\tfrac{1}{2} u^t \varkappa u\big) \theta[K](\omega^{-1} u;  \omega^{-1} \omega'),
\end{equation}
where $[K]$ denotes the characteristic of the vector $K$ of Riemann constants,  
a symmetric matrix $\varkappa = \eta \omega^{-1}$ is obtained from the second 
kind period matrix $\eta$ defined by \eqref{etaM}, and the constant $C$ does not depend of $u$.

The origin $u=0$ of $\Jac(\mathcal{C})$ is the Abel image of infinity on $\mathcal{C}$,
which also serves as the neutral point.
Every point $u$ in the fundamental domain of $\Jac(\mathcal{C})$ 
is represented by its characteristic $[\varepsilon]$, namely
\begin{equation}\label{uChar}
u[\varepsilon] =  \omega \varepsilon +   \omega' \varepsilon'.
\end{equation}

\subsection{Uniformization of the curve}
Every class of equivalent divisors on $\mathcal{C}$ has a representative
$P_1 + P_2 - 2 \infty$, and $P_1$, $P_2$ are not in involution, that is $P_2 \neq -P_1$.
Such a representative is called a \emph{reduced divisor}. Since the poles are located at infinity,
which serves as the basepoint, it is convenient to define every reduced divisor by its positive part,
as follows
\begin{itemize}
\item non-special $D_2 = (x_1,y_1)+(x_2,y_2)$ of degree $2$,
\item special $D_1=(x_1,y_1) + \infty$ of degree $1$
\item neutral $O = 2\infty$ of degree $0$, $u(2\infty)=0$.
\end{itemize}
Let $\mathfrak{C}_2$ be a collection of all degree $2$ non-special divisors.
As mentioned in Introduction, we denote by  
$(\sigma)_0 = \{u \in \Jac(\mathcal{C}) \mid \sigma(u)=0\}$  the theta divisor.
Then $\mathcal{A}(\mathfrak{C}_2) = \Jac(\mathcal{C})\backslash (\sigma)_0$.

Uniformization of $\mathcal{C}$ is given in terms of $\wp$-functions,
known as multiply periodic after \cite{bakerMP}, and Kleinian after \cite{belHKF}.
Actually,
\begin{gather*}
\wp_{i,j}(u) = -\frac{\partial^2 \log \sigma(u) }{\partial u_i \partial u_j },\qquad
\wp_{i,j,k}(u) = -\frac{\partial^3 \log \sigma(u) }{\partial u_i \partial u_j \partial u_k},\qquad
\text{etc}.
\end{gather*}
Meromorphic functions $\wp_{i,j}$, $\wp_{i,j,k}$, and all higher derivatives
are $\{\omega,\omega'\}$-periodic.

\begin{prop}[\textbf{The Jacobi inversion problem}]\label{P:JIPr}  \cite[\S\,216]{bakerAF}
Given $u\in \Jac(\mathcal{C})\backslash (\sigma)_0$,
 the Abel pre-image  of $u$
is a non-special divisor  $D_2 = (x_1,y_1)+(x_2,y_2)$
with coordinates uniquely defined by the system
\begin{equation}\label{JacIvnPrD2}
\mathcal{R}_{4}(x;u) = 0,\qquad\qquad \mathcal{R}_{5}(x,y;u) =0,\quad
\end{equation}
where $u = \mathcal{A}(D_2)$, 
\begin{equation}\label{wpJacIvnPr}
\begin{split}
&\mathcal{R}_{4}(x;u) \equiv x^2 - x \wp_{1,1}(u) - \wp_{1,3}(u), \\
&\mathcal{R}_{5}(x,y;u) \equiv y + \tfrac{1}{2} x \wp_{1,1,1}(u) + \tfrac{1}{2} \wp_{1,1,3}(u). 
\end{split}
\end{equation}
In other words, $D_2$ is a common divisor of zeros of the two polynomial 
functions  $\mathcal{R}_{4}$, $\mathcal{R}_{5}$ on the curve $\mathcal{C}$.
\end{prop}

\begin{prop}\label{P:JIPRels}
$\mathcal{R}_{4}$, $\mathcal{R}_{5}$ defined by \eqref{wpJacIvnPr} 
on the common divisor of zeros are 
connected by the relation
\begin{equation}\label{R4R5Rel}
\mathcal{R}_{5}(x_i,y_i;u) = - \partial_{u_1} \mathcal{R}_{4}(x_i;u) , \quad i=1,\,2.
\end{equation}
\end{prop}
\begin{proof}
If $D_2 =  (x_1,y_1)+(x_2,y_2)$ is a common divisor of zeros of $\mathcal{R}_{4}$, $\mathcal{R}_{5}$, then
$\mathcal{R}_{4}(x_i;u) = 0$, $i=1$, $2$. Differentiating with respect to $u_1$, and taking into account
that the Jacobian matrix of $\mathcal{A}^{-1}: \mathfrak{C}_2 \to \Jac(\mathcal{C})\backslash (\sigma)_0$ 
has the entries
\begin{align}\label{dxduJ}
&\frac{\partial x_1}{\partial u_1} =  \frac{-2y_1}{x_1-x_2},&
&\frac{\partial x_2}{\partial u_1} =  \frac{2y_2}{x_1-x_2},& 
&\frac{\partial x_1}{\partial u_3} =  \frac{2 x_2 y_1}{x_1-x_2},&
&\frac{\partial x_2}{\partial u_3} =  \frac{-2 x_1 y_2}{x_1-x_2},&
\end{align}
we immediately obtain \eqref{R4R5Rel}.
\end{proof}

\begin{rem}\label{R:JIPr}
In general, $\mathcal{R}_{4}$, $\mathcal{R}_{5}$ from Proposition~\ref{P:JIPr} have the form
\begin{gather}\label{D2PolyDef}
\mathcal{R}_{4}(x) = x^2 + \alpha_2 x + \alpha_4, \qquad
\mathcal{R}_{5}(x,y) = y  + \beta_3 x + \beta_5.
\end{gather}
Let $D_2 = (x_1,y_1) + (x_2,y_2)$ be the common divisor of zeros of $\mathcal{R}_{4}$ and
$\mathcal{R}_5$, then
\begin{subequations}\label{AlphaBeta}
\begin{align}
&\alpha_2  = - (x_1 + x_2) = - \wp_{1,1}(u),&
&\alpha_4  = x_1 x_2 = - \wp_{1,3}(u),& \\
&\beta_3 = - \frac{y_1-y_2}{x_1-x_2} = \tfrac{1}{2}\wp_{1,1,1}(u), &
&\beta_5 = \frac{x_2 y_1 - x_1y_2}{x_1-x_2} = \tfrac{1}{2}\wp_{1,1,3}(u). &
\end{align}
\end{subequations}
Actually, $\alpha_2$, $\alpha_4$, $-\beta_3$, $-\beta_5$ are known 
as the Mumford coordinates,
and the pair of polynomials in $x$:  $\mathcal{R}_{4}(x)$, $y - \mathcal{R}_{5}(x,y)$,
are the first two from the triple of Mumford's representation of $D_2$. 
In what follows, we call $\alpha_2$, $\alpha_4$, $\beta_3$,
$\beta_5$ the \emph{Mumford coordinates}, for simplicify. Instead of Mumford's representation,
we use  polynomial functions $\mathcal{R}_{4}$, $\mathcal{R}_5$,
which are sufficient to define any  divisor from $\mathfrak{C}_2$ uniquely.
\end{rem}

From \cite[Theorem 3.2]{belHKF} we have the fundamental cubic relations on $\mathcal{C}$.
After eliminating $\wp_{3,3}(u)$ an algebraic model of $\Jac(\mathcal{C})\backslash (\sigma)_0$ is obtained.
\begin{prop}
Given $u\in \Jac(\mathcal{C})\backslash (\sigma)_0$, 
the following identities define $\Jac(\mathcal{C}) \backslash (\sigma)_0$:
\begin{gather}\label{wpJacEqs}
\begin{split}
&\tfrac{1}{2} \wp_{1,1,1}(u) \wp_{1,1,3}(u) + \tfrac{1}{4}\wp_{1,1}(u)\wp_{1,1,1}^2(u) = 
  2 \wp_{1,1}^2(u) \wp_{1,3}(u) \\
&\qquad + \wp_{1,3}^2(u) + 2 \lambda_2 \wp_{1,1}(u) \wp_{1,3}(u)
+  \lambda_4 \wp_{1,3}(u) + \lambda_8 \\
&\qquad + \wp_{1,1}(u) \big(\wp_{1,1}(u) ^3 + \wp_{1,1}(u) \wp_{1,3}(u) 
+ \lambda_2 \wp_{1,1}(u)^2 + \lambda_4 \wp_{1,1}(u) + \lambda_6\big),  \\
&\tfrac{1}{4} \wp_{1,1,3}(u)^2 + \tfrac{1}{4} \wp_{1,3}(u)\wp_{1,1,1}^2(u) =
 \wp_{1,1}(u) \wp_{1,3}(u)^2 + \lambda_2 \wp_{1,3}(u)^2 + \lambda_{10} \\
&\qquad +  \wp_{1,3}(u) \big(\wp_{1,1}(u) ^3 + \wp_{1,1}(u) \wp_{1,3}(u) 
+ \lambda_2 \wp_{1,1}(u)^2 + \lambda_4 \wp_{1,1}(u) + \lambda_6\big).
\end{split}
\end{gather}
\end{prop}
This model is used in \cite[Theorem 2.8]{Uchida}.
The four functions  $\wp_{1,1}$, $\wp_{1,3}$, $\wp_{1,1,1}$, $\wp_{1,1,3}$, 
which arise from \eqref{wpJacIvnPr}, serve as  coordinates on 
$\Jac(\mathcal{C})\backslash (\sigma)_0$.
The differential field of meromorphic functions on $\Jac(\mathcal{C})\backslash (\sigma)_0$ is
$\Complex [\wp_{1,1},\,\wp_{1,3},\,\wp_{1,1,1}$, $\wp_{1,1,3}]$, that is, consists of
 polynomial expressions in these four functions, see \cite{bel08}.
For example, 
\begin{subequations}
\begin{align}
&\wp_{1,3,3}(u) = \wp_{1,3}(u) \wp_{1,1,1}(u) - \wp_{1,1}(u) \wp_{1,1,3}(u), \label{wp133}\\
&\wp_{1,1,1,1}(u) = 6 \wp_{1,1}(u)^2 + 4 \wp_{1,3}(u) + 4 \lambda_2  \wp_{1,1}(u) + 2\lambda_4,\\
&\wp_{1,1,1,3}(u) = 6 \wp_{1,3}(u) \wp_{1,1}(u) - 2 \wp_{3,3}(u) + 4 \lambda_2 \wp_{1,3}(u),\\
&\wp_{3,3}(u) = \tfrac{1}{4} \wp_{1,1,1}(u)^2 - \wp_{1,1}(u)^3 - \wp_{1,3}(u) \wp_{1,1}(u) \\
&\phantom{\wp_{3,3}(u) =} - \lambda_2 \wp_{1,1}(u)^2  - \lambda_4 \wp_{1,1}(u) - \lambda_6. \notag
\end{align}
\end{subequations}

According to Proposition\;\ref{P:JIPr}, the map 
$$u \mapsto \big(\wp_{1,1}(u),\,\wp_{1,3}(u),\,\wp_{1,1,1}(u),\,\wp_{1,1,3}(u)\big)
= \big({-}\alpha_2,\, {-}\alpha_4,\, 2 \beta_3,\, 2\beta_5 \big) $$
takes $u\in \Jac(\mathcal{C})\backslash (\sigma)_0$ to $\mathfrak{C}_2$.
In terms of the Mumford coordinates, \eqref{wpJacEqs} acquire the form
\begin{gather}\label{ABJacEqs}
\begin{split}
&J_8(\alpha_2,\alpha_4,\beta_3,\beta_5;\lambda) 
\equiv 2\beta_3 \beta_5 - \alpha_2^2 \alpha_4  - \alpha_4^2  + \lambda_4 \alpha_4 -\lambda_8 \\
&\qquad - \alpha_2 \big(\beta_3^2 + \alpha_2^3 - 4 \alpha_2 \alpha_4 + \lambda_2 (2 \alpha_4 - \alpha_2^2) 
+ \lambda_4 \alpha_2 - \lambda_6 \big) = 0,  \\
&J_{10}(\alpha_2,\alpha_4,\beta_3,\beta_5;\lambda) 
\equiv \beta_5^2 - 2 \alpha_2 \alpha_4^2 + \lambda_2 \alpha_4^2 - \lambda_{10} \\
&\qquad - \alpha_4 \big(\beta_3^2 + \alpha_2^3 - 4 \alpha_2 \alpha_4 + \lambda_2 (2 \alpha_4 - \alpha_2^2) 
+ \lambda_4 \alpha_2 - \lambda_6 \big) = 0.
\end{split}
\end{gather}

\section{Polynomial functions on a curve}
Since the base point is fixed at infinity, divisors are described by means of polynomial functions
on $\mathcal{C}$, which form a ring $\mathfrak{P}(\mathcal{C}) = \Complex[x,y] / f(x,y;\lambda)$. 
Recall, that each  divisor is represented by its positive part.

Let $\mathcal{R}_{\mathfrak{w}}$ be a polynomial function 
of weight $\mathfrak{w}$ from $\mathfrak{P}(\mathcal{C})$. The divisor of zeros $(\mathcal{R}_{\mathfrak{w}})_0$ 
is of degree $\mathfrak{w}$, and defined by the system
\begin{equation}\label{ZeroDivRn}
\mathcal{R}_{\mathfrak{w}}(x,y) =0 ,\qquad\qquad f(x,y; \lambda)=0.
\end{equation}
Polynomial functions are constructed from monomials $x^i y^j$, arranged ascendingly by the Sato weight
into an ordered list $\mathfrak{M}$, see for example \eqref{MlListC25} in the case of
a genus two curve.

\begin{prop}
A polynomial function $\mathcal{R}_{\mathfrak{w}}$ of weight $\mathfrak{w}$
is constructed from monomials $\{\mathfrak{m}_{\widetilde{\mathfrak{w}}} \in \mathfrak{M} \mid 
\widetilde{\mathfrak{w}} \leqslant \mathfrak{w} \}$, namely
$$  \mathcal{R}_{\mathfrak{w}}(x,y) = \sum_{\widetilde{\mathfrak{w}} \leqslant \mathfrak{w}}
c_{\widetilde{\mathfrak{w}}}
 \mathfrak{m}_{\widetilde{\mathfrak{w}}}.$$
 \end{prop}

\begin{prop}\label{P:PolyFunctDef}
A polynomial function $\mathcal{R}_{\mathfrak{w}}$ of weight $\mathfrak{w} \geqslant 2g$ 
on  a genus $g$ curve~$\mathcal{C}$
is uniquely defined by a positive divisor $D_{\mathfrak{w}-g}$ of degree $\mathfrak{w}-g$ 
such that  $D_{\mathfrak{w}-g} \subset (\mathcal{R}_{\mathfrak{w}})_0$, 
and $D_{\mathfrak{w}-g}$ contains no  groups of points in involution
(repeated points are allowed).
The function $\mathcal{R}_{\mathfrak{w}}$ is constructed as a linear combination of 
the first $\mathfrak{w}-g+1$ monomials from the ordered list $\mathfrak{M}$.
\end{prop}
\begin{proof}
If $\mathfrak{w} \geqslant 2g$, there exist $\mathfrak{w} -g+1$ such monomials, 
since $g$ weights between $0$ and $2g-1$ belong to
the Weierstrass gap sequence. 
A monic polynomial composed as a linear combination of 
 monomials of weights up to $\mathfrak{w}$
has $\mathfrak{w}-g$ degrees of freedom. Thus,
a positive divisor $D_{\mathfrak{w}-g} =\sum_{k=1}^{\mathfrak{w}-g} (x_k,y_k)$
with no groups of points in involution defines $\mathcal{R}_{\mathfrak{w}}$ uniquely.
Indeed, $\mathcal{R}_{\mathfrak{w}}$ is constructed from  $D_{\mathfrak{w}-g}$ with all distinct points
as follows
\begin{subequations}\label{DetR}
\begin{equation}\label{DetRD}
 \mathcal{R}_{\mathfrak{w}}(x,y) = 
 \frac{\small \begin{vmatrix}  
 \mathfrak{m}_{\mathfrak{w}}(x,y) & \mathfrak{m}_{\mathfrak{w}-1}(x,y) 
 & \dots & \mathfrak{m}_0(x,y)  \\
 \mathfrak{m}_{\mathfrak{w}}(x_1,y_1) 
 & \mathfrak{m}_{\mathfrak{w}-1}(x_1,y_1) & \dots &  \mathfrak{m}_0(x_1,y_1) \\
 \vdots & \ddots & \vdots\\
 \mathfrak{m}_{\mathfrak{w}}(x_{\mathfrak{w}-g },y_{\mathfrak{w}-g}) 
 & \mathfrak{m}_{\mathfrak{w}-1}(x_{\mathfrak{w}-g},y_{\mathfrak{w}-g}) & \dots &  
 \mathfrak{m}_0(x_{\mathfrak{w}-g},y_{\mathfrak{w}-g}) 
  \end{vmatrix}}
 {\small \begin{vmatrix}  
 \mathfrak{m}_{\mathfrak{w}-1}(x_1,y_1) & \dots &  \mathfrak{m}_0(x_1,y_1) \\
 \vdots & \ddots & \vdots\\
 \mathfrak{m}_{\mathfrak{w}-1}(x_{\mathfrak{w}-g},y_{\mathfrak{w}-g}) & \dots &  
 \mathfrak{m}_0(x_{\mathfrak{w}-g},y_{\mathfrak{w}-g}) \end{vmatrix}}.
\end{equation}
If some points coincide, say  $P_i =P_1$, $i=2$, \ldots $n$, then row $i+1$ in the numerator and 
row $i$ in the denominator of \eqref{DetRD} are replaced with 
\begin{equation}\label{DetRM}
\Big(\frac{\rmd^{i-1} }{\rmd x^{i-1}} \mathfrak{m}_{\widetilde{\mathfrak{w}}}(x,y(x)) 
\Big|_{\substack{x=x_1\\ y(x_1)=y_1} }\Big).
\end{equation}
\end{subequations}

Therefore, the determinant formula \eqref{DetR} produces a monic function,
which represents $\mathcal{R}_{\mathfrak{w}}$ with 
$D_{\mathfrak{w}-g} \subset (\mathcal{R}_{\mathfrak{w}})_0$  uniquely. 
\end{proof}

\begin{cor}\label{C:Rwl2g}
On a genus $g$ hyperelliptic curve, $\mathcal{R}_{\mathfrak{w}}$ of weights $\mathfrak{w} \leqslant 2g$
are polynomials in $x$ only, and have even weights
$\mathfrak{w} = 2\mathfrak{k}$, $\mathfrak{k}\leqslant g$.
Moreover, $\mathcal{R}_{2\mathfrak{k}}$ is uniquely defined by a set of distinct
$\{x_i \mid i=1,\dots \mathfrak{k}\}$, namely
\begin{gather*}
\mathcal{R}_{2\mathfrak{k}}(x) = \prod_{i=1}^{\mathfrak{k}} (x-x_i),
\end{gather*}
and  $(\mathcal{R}_{\mathfrak{w}})_0$ consists of pairs of points connected by involution:
\begin{gather*}
(\mathcal{R}_{2\mathfrak{k}})_0 = \textstyle \sum_{i=1}^{\mathfrak{k}} \big(\big(x_i,y(x_i) )+ (x_i,-y(x_i))\big).
\end{gather*}
\end{cor}
\begin{proof}
The Weierstrass gap sequence on a hyperelliptic curve of genus $g$
has the form $\mathfrak{W} = \{2i-1\mid i=1,\dots g\}$, and monomials of weights $\mathfrak{w} \leqslant 2g$
 have the form $x^k$, $0 \leqslant k \leqslant g$.
Thus, there exist no polynomial functions of odd weights $\mathfrak{w} = 2i-1$, $i=1$, \ldots, $g$,
and all functions of weights $\mathfrak{w} \leqslant 2g$ are polynomials in $x$ only.
\end{proof}

\begin{cor}\label{C:Rwg2gDecomp}
Let $\mathcal{R}_{\mathfrak{w}}$ be a polynomial function of weight $\mathfrak{w} > 2g$
 on a hyperelliptic curve of genus $g$, and $D_{\mathfrak{w}-g} \subset (\mathcal{R}_{\mathfrak{w}})_0$ 
 with $\deg D_{\mathfrak{w}-g} = \mathfrak{w} - g$.
Then 
\begin{itemize}
\item $\mathcal{R}_{\mathfrak{w}}$
 is indecomposable, if $D_{\mathfrak{w}-g}$ contains no pairs of points connected by involution, 
 \item $\mathcal{R}_{\mathfrak{w}}$ has a factor $(x-x_i)$, if $D_{\mathfrak{w}-g}$ contains
 a pair of points connected by involution: $(x_i,y_i)$, $(x_i,-y_i)$.
\end{itemize}
\end{cor}
\begin{proof}
Suppose, $D_{\mathfrak{w}-g}$ contains a pair of points connected by involution,
say $P_1=(x_1,y_1)$, $P_2 = (x_1,-y_1)$. We use the formula \eqref{DetR} to construct
$\mathcal{R}_{\mathfrak{w}}$.
In the numerator we do the following operations with rows 2 and 3:
\begin{align*}
 &(\mathfrak{m}_{\widetilde{\mathfrak{w}}}(x_1,y_1) ) 
\mapsto  (\mathfrak{m}_{\widetilde{\mathfrak{w}}}(x_1,y_1) + \mathfrak{m}_{\widetilde{\mathfrak{w}}}(x_1,-y_1) ),\\
 &(\mathfrak{m}_{\widetilde{\mathfrak{w}}}(x_1,-y_1) ) 
\mapsto  (\mathfrak{m}_{\widetilde{\mathfrak{w}}}(x_1,-y_1) - \mathfrak{m}_{\widetilde{\mathfrak{w}}}(x_1,y_1) )
\end{align*}
Then the first three rows acquire the form
\begin{gather*}
\left| \begin{array}{ccccccccccc}
1 & x & \dots & x^g & y & x^{g+1} & y x & \dots & x^{k+g+1} & y x^k  & \dots \\
2 & 2x_1 & \dots & 2x_1^g & 0 & 2x_1^{g+1} & 0 & \dots & 2x_1^{k+g+1} & 0 & \dots \\
0 & 0 & \dots & 0 & -2y_1 & 0 & -2 y_1 x_1 & \dots & 0 & -2y_1 x_1^k  & \dots \\
\vdots & \vdots & \vdots & \vdots & \vdots & \vdots & \vdots & \vdots & \vdots & \vdots & \vdots 
\end{array} \right| \sim \\
-4 y_1 \left| \begin{array}{ccccccccccc}
1 & x & \dots & x^g & y & x^{g+1} & y x & \dots & x^{k+g+1} & y x^k  & \dots \\
1 & x_1 & \dots & x_1^g & 0 & x_1^{g+1} & 0 & \dots & x_1^{k+g+1} & 0 & \dots \\
0 & 0 & \dots & 0 & 1 & 0 & x_1 & \dots & 0 & x_1^k  & \dots \\
\vdots & \vdots & \vdots & \vdots & \vdots & \vdots & \vdots & \vdots & \vdots & \vdots & \vdots
\end{array} \right|.
\end{gather*}
Subtracting row 2, and row 3 multiplied by $y$ from row 1, we extract
the common multiple $(x-x_1)$, which represents the pair of points in involution.
\end{proof}

Inversion and addition of divisors are conveniently implemented by means of polynomial functions
from $\mathfrak{P}(\mathcal{C})$,
see \cite{bl2005}.
\begin{theo}
The inversion of a non-special divisor $D_{g}$ of degree $g$  
on a  curve  of genus~$g$ is defined by
the polynomial function $\mathcal{R}_{2g}$ of weight $2g$ 
with $D_{g} \subset (\mathcal{R}_{2g})_0$.
 \end{theo}
 \begin{proof}
Let a  degree $g$ divisor $D^\ast_{g}$ be the complement  of $D_{g}$ in $(\mathcal{R}_{2g})_0$, that is
$(\mathcal{R}_{2g})_0 =  D_{g}  + D^\ast_{g}$. Since $(\mathcal{R}_{2g})_0 \sim 2g \infty = O$,
 $D^\ast_{g}$ serves as the inverse of $D_{g}$, that is $D^\ast_{g} = - D_{g} $.
 \end{proof}

\begin{theo}
The inversion of a special divisor $D_{\mathfrak{k}}$ of degree  $\mathfrak{k} < g$
on a hyperelliptic curve of genus~$g$ is defined by
the polynomial function $\mathcal{R}_{2\mathfrak{k}}$ of weight $2\mathfrak{k}$ 
with $D_{\mathfrak{k}} \subset (\mathcal{R}_{2\mathfrak{k}})_0$.
 \end{theo}
\begin{proof}
Let $D_{\mathfrak{k}} = \sum_{k=1}^{\mathfrak{k}} (x_k,y_k)$.
From this divisor a polynomial function is constructed as a linear combination of the first $\mathfrak{k}$
monomials from the ordered list $\mathfrak{M}$. Such monomials are $1$, $x$, \ldots, $x^{\mathfrak{k}-1}$, if
the curve is hyperelliptic. Thus, we obtain a funciton
$\mathcal{R}_{2\mathfrak{k}}$ of weight $2\mathfrak{k}$, which is a polynomial in $x$ only; and the 
complement of $D_{\mathfrak{k}}$ is $-D_{\mathfrak{k}}$, as follows from
 Corollary~\ref{C:Rwl2g}.
\end{proof}

\begin{theo}\label{T:Add}
The addition  of two degree $g$ non-special divisors $D_g$, $\tilde{D}_g$ 
on  a  curve of genus~$g$ is defined by the
polynomial function $\mathcal{R}_{3g}$  of weight $3g$
with $D_{g} + \tilde{D}_g \subset (\mathcal{R}_{3g})_0$.
 \end{theo}
\begin{proof}
The required function $\mathcal{R}_{3g}$ is constructed from $D_g$ and $\tilde{D}_g$,
according to Proposition~\ref{P:PolyFunctDef}. Then the complement divisor $D_g^\ast$
such that $D_{g} + \tilde{D}_g + D_g^\ast = (\mathcal{R}_{3g})_0$ is the inverse of $D_{g} + \tilde{D}_g$,
and so $-D_g^\ast$ is the reduced divisor equivalent to  the sum $D_{g} + \tilde{D}_g$.
\end{proof}

\begin{theo}
The addition  of two  divisors $D_m$, $D_n$ of degrees $m$ and $n$ 
on a hyperelliptic curve of genus~$g$ is defined by the
polynomial function $\mathcal{R}_{m+n+g}$  of weight $m+n+g$
with $D_m + D_n\subset (\mathcal{R}_{3g})_0$, if $m+n\geqslant g$
and $D_m \,{+}\, D_n$ contains no points in involution.
 \end{theo}
A proof is similar to the proof of Theorem~\ref{T:Add}.

\begin{theo}\label{P:NSDgdef}
A non-special divisor $D_g$ of degree $g$ on a hyperelliptic curve of genus $g$ 
 is uniquely defined by a pair of functions $\mathcal{R}_{2g}$, $\mathcal{R}_{2g+1}$.
\end{theo}
\begin{proof}
The function $\mathcal{R}_{2g}$ is defined by $x$-coordinates of the support of $D_g$.
The same divisor $D_g$ can be used to construct $\mathcal{R}_{2g+1}$, according to 
Proposition~\ref{P:PolyFunctDef}, with an addition condition: the coefficient of monomial $\mathfrak{m}_{2g}$
vanishes.
Such a pair of functions define a solution of the Jacobi inversion problem, see Proposition~\ref{P:JIPr}.
\end{proof}

\begin{rem}
If the two functions $\mathcal{R}_{2g}$, $\mathcal{R}_{2g+1}$
define $D_{g}$, then $(\mathcal{R}_{2g})_0 \,{=}\, D_{g} + D_{g}^{\ast}$,
where $D_{g}^{\ast} = - D_{g}$, 
and  $(\mathcal{R}_{2g+1})_0 = D_{g} + D_{g+1}$. We also have
${-}D_{g} \,{+}\, ({-} D_{g+1}) = (\mathcal{R}_{2g+1}^{-})_0$, where $\mathcal{R}_{2g+1}^{-}$
denotes the function $\mathcal{R}_{2g+1}(x,-y)$.
\end{rem}

\section{Addition and duplication}
Below, the addition and duplication laws are derived according to \cite{bl2005}.

\subsection{Addition law}
Addition of two degree $2$ non-special divisors is defined by a polynomial function of weight~$6$,
which has the form
$$\mathcal{R}_6(x,y) = x^3 + \gamma_1 y + \gamma_2 x^2 + \gamma_4 x + \gamma_6.$$
Let $\mathcal{R}_6$ be defined by  $D_4 \subset (\mathcal{R}_6)_0$ such that 
$D_4 = D_{2P} + D_{2Q} = (P_1 + P_2) + (Q_1+Q_2)$
without repeated\footnote{There is no $n$-torsion divisors which contain repeated points. See the definition of 
an $n$-torsion divisor in the next section.} 
points, see Proposition~\ref{P:PolyFunctDef}. 
According to Proposition~~\ref{P:JIPr} and 
Remark~\ref{R:JIPr}, let $D_{2A}$ be defined by $\mathcal{R}^{[A]}_{4}$ and $\mathcal{R}^{[A]}_{5}$
with Mumford coordinates $\alpha^{[A]}_2$, $\alpha^{[A]}_4$,
$\beta^{[A]}_3$, $\beta^{[A]}_5$, where $A$ stands for $P$, or $Q$.

With $A=P$, $Q$  we have the equality
\begin{equation}\label{R6inR4R5}
\mathcal{R}_6(x,y) = \gamma_1 \mathcal{R}_5^{[A]}(x,y) 
+ (x-\alpha^{[A]}_2 +\gamma_2) \mathcal{R}_4^{[A]}(x)
\equiv \textsf{S}_4^{[A]} x + \textsf{S}_6^{[A]},
\end{equation}
which introduces $\textsf{S}_4^{[A]}$, $\textsf{S}_6^{[A]}$ as
polynomials in  $\alpha^{[A]}_2$, $\alpha^{[A]}_4$,
$\beta^{[A]}_3$, $\beta^{[A]}_5$,
and coefficients $\gamma_k$ of $\mathcal{R}_6$.
Since $\mathcal{R}_6$ vanishes on $D_{2P}$ and $D_{2Q}$, we 
obtain four equations, 
which are linear in $\gamma_k$, and admit the matrix form
\begin{gather}\label{GammaInABEqs}
\begin{pmatrix}
\textsf{S}_6^{[P]} \\ \textsf{S}_4^{[P]} \\ \textsf{S}_6^{[Q]} \\ \textsf{S}_4^{[Q]} 
\end{pmatrix} \equiv
\begin{pmatrix}
1 & 0 & -\alpha_4^{[P]} & -\beta_5^{[P]} \\
0 & 1 & -\alpha_2^{[P]} & -\beta_3^{[P]} \\
1 & 0 & -\alpha_4^{[Q]} & -\beta_5^{[Q]} \\
0 & 1 & -\alpha_2^{[Q]} & -\beta_3^{[Q]}
\end{pmatrix}
\begin{pmatrix}
\gamma_6 \\ \gamma_4 \\ \gamma_2 \\ \gamma_1
\end{pmatrix}
+  \begin{pmatrix}
\alpha_2^{[P]} \alpha_4^{[P]} \\ (\alpha^{[P]}_2)^2 - \alpha^{[P]}_4 \\ 
\alpha_2^{[Q]} \alpha_4^{[Q]} \\ (\alpha^{[Q]}_2)^2 - \alpha^{[Q]}_4
\end{pmatrix}
= 0.
\end{gather}
Solving \eqref{GammaInABEqs}, we find
\begin{subequations}\label{gammaD2}
\begin{align}
 \begin{pmatrix} \gamma_2 \\ \gamma_1 \end{pmatrix} 
& =  \begin{pmatrix} 
\alpha^{[P]}_4 -  \alpha^{[Q]}_4 &  \beta^{[P]}_5 - \beta^{[Q]}_5 \\
\alpha^{[P]}_2 - \alpha^{[Q]}_2  & \beta^{[P]}_3 - \beta^{[Q]}_3 
 \end{pmatrix}^{-1}
 \begin{pmatrix} 
 \alpha^{[P]}_2 \alpha^{[P]}_4 - \alpha^{[Q]}_2 \alpha^{[Q]}_4\\
 (\alpha^{[P]}_2)^2 - (\alpha^{[Q]}_2)^2 - \alpha^{[P]}_4 + \alpha^{[Q]}_4
 \end{pmatrix},\\
 \begin{pmatrix} \gamma_6 \\ \gamma_4 \end{pmatrix} 
&=  \begin{pmatrix} \alpha_4^{[P]} & \beta_5^{[P]} \\
 \alpha_2^{[P]} & \beta_3^{[P]}  \end{pmatrix} 
 \begin{pmatrix} \gamma_2 \\ \gamma_1 \end{pmatrix} -
  \begin{pmatrix} \alpha^{[P]}_2 \alpha^{[P]}_4\\
  (\alpha^{[P]}_2)^2 - \alpha^{[P]}_4  \end{pmatrix}.
\end{align}
\end{subequations}
Formulas \eqref{gammaD2} define $\mathcal{R}_6$ in terms of 
$\alpha^{[A]}_2$, $\alpha^{[A]}_4$, $\beta^{[A]}_3$, $\beta^{[A]}_5$ with $A=P$, $Q$.

Let $D_2^\ast$ be a divisor of degree $2$  such that $(\mathcal{R}_6)_0 = D_{2P} + D_{2Q} + D_2^\ast$.
According to Remark~\ref{R:JIPr}, we define $D_2^\ast$  by the two polynomial functions
\begin{equation*}
\mathcal{R}_4^\ast(x) = x^2 + \alpha_2^\ast x + \alpha_4^\ast, \qquad
\mathcal{R}_5^\ast(x,y) = y + \beta_3^\ast x + \beta_5^\ast.
\end{equation*}

Recalling that $(\mathcal{R}_6)_0$ is defined by a system of the form \eqref{ZeroDivRn},
we have
\begin{equation}\label{HfEq}  
-\gamma_1^2 f(x,y_{\mathcal{R}_6}; \lambda) = 
\mathcal{R}^{[P]}_{4}(x) \mathcal{R}^{[Q]}_{4}(x) \mathcal{R}_4^\ast(x),
\end{equation}
where $y_{\mathcal{R}_6}$ is obtained from $\mathcal{R}_6(x,y_{\mathcal{R}_6})=0$, namely
$$ y_{\mathcal{R}_6} = {-}  \gamma_1^{-1} \big(x^3 + \gamma_2 x^2 + \gamma_4 x + \gamma_6\big).$$
Coordinates  $\alpha_2^\ast$, $\alpha_4^\ast$ are computed 
from  coefficients of $x^5$ and $x^4$ in \eqref{HfEq}.
Then $\mathcal{R}_{5}^\ast$ is derived from \eqref{R6inR4R5},
since $\mathcal{R}_{6}$  vanishes on $D_2^\ast$ as well, namely
\begin{equation}
\mathcal{R}_{5}^\ast(x,y) = \gamma_1^{-1}
\big( \mathcal{R}_6(x,y) - (x + \gamma_2 - \alpha_2^\ast) \mathcal{R}_4^\ast(x) \big).
\end{equation}
This produces coordinates $\beta_3^\ast$, $\beta_5^\ast$.

Finally, let  $\widetilde{D}_2 \equiv - D_2^\ast$, and so $\widetilde{D}_2 \sim D_{2P} + D_{2Q}$,
that is $\widetilde{D}_2$ is the reduced divisor equivalent to $D_{2P} + D_{2Q}$.
Then $\widetilde{D}_2$ is defined by the Mumford coordinates
\begin{subequations}\label{DivSumAlphaBeta}
\begin{align}
\begin{split}\label{DivSumAlpha}
\alpha_2^{[P+Q]} &= \alpha_2^\ast = - \alpha_2^{[P]} - \alpha_2^{[Q]} + 2\gamma_2 - \gamma_1^2, \\
\alpha_4^{[P+Q]} &= \alpha_4^\ast = - \alpha_4^{[P]} - \alpha_4^{[Q]} + \big(\alpha_2^{[P]}\big)^2  + \alpha_2^{[P]} \alpha_2^{[Q]} + \big(\alpha_2^{[Q]}\big)^2 \\
&\phantom{\alpha_4^{[P,Q]} = }
- \big(\alpha_2^{[P]} + \alpha_2^{[Q]}\big)\big(2\gamma_2 - \gamma_1^2 \big) 
 + 2 \gamma_4 + \gamma_2^2  - \lambda_2 \gamma_1^2,
\end{split}\\
\begin{split}\label{DivSumBeta}
\beta_3^{[P+Q]} &= -\beta_3^\ast = - \gamma_1^{-1} \big( 
  (\alpha_2^{[P+Q]})^2 - \alpha_4^{[P+Q]}  - \gamma_2 \alpha_2^{[P+Q]} + \gamma_4\big),\\
\beta_5^{[P+Q]} &= -\beta_5^\ast = - \gamma_1^{-1} \big(
 \alpha_2^{[P+Q]} \alpha_4^{[P+Q]} -  \gamma_2 \alpha_4^{[P+Q]} + \gamma_6 \big).
\end{split}
\end{align}
\end{subequations}

\subsection{Addition law on special divisors}
The case of adding a special divisor $D_{1Q} = Q_1 = (x^{[Q]}_1,y^{[Q]}_1)$ to a non-special divisors $D_{2P} = P_1+P_2$ 
can be obtained from the above formulas \eqref{DivSumAlphaBeta} by taking the limit as $Q_2 \to \infty$,
that is by applying the parametrization \eqref{param} to $(x^{[Q]}_2,y^{[Q]}_2)$
and taking the limit as $\xi \to 0$. As a result, the following formulas are obtained 
($x^{[Q]}_1 \equiv x_Q$, $y^{[Q]}_1 \equiv y_Q$):
\begin{equation}\label{DgPABdef}
\begin{split}
&\alpha_2^{[P+Q]} = - \alpha_2^{[P]}  + x_Q + \lambda_2 - \gamma_1^2,\\
&\alpha_4^{[P+Q]} = - \alpha_4^{[P]}   + (\alpha_2^{[P]} )^2 + 
(x_Q-\alpha_2^{[P]} ) \big(x_Q + \lambda_2 - \gamma_1^2\big)
+ \lambda_4 - 2 \gamma_1 \gamma_3,\\
&\beta_3^{[P+Q]} =  \gamma_1 \alpha_2^{[P+Q]}  - \gamma_3, \\
&\beta_5^{[P+Q]} = \gamma_1 \alpha_4^{[P+Q]}  - \gamma_5, \\
\end{split}
\end{equation}
where
\begin{align*}
&\gamma_1 = - \frac{y_Q + x_Q \beta_3^{[P]} + \beta_5^{[P]}}
{x_Q^2 + x_Q \alpha_2^{[P]} + \alpha_4^{[P]}},\\
&\gamma_3 = \frac{- y_Q \alpha_2^{[P]} + x_Q^2 \beta_3^{[P]} + \alpha_4^{[P]} \beta_3^{[P]}  
- \alpha_2^{[P]} \beta_5^{[P]}}
{x_Q^2 + x_Q \alpha_2^{[P]} + \alpha_4^{[P]}},\\
&\gamma_5 = \frac{- y_Q \alpha_4^{[P]} + x_Q^2 \beta_5^{[P]} - x_Q (\alpha_4^{[P]} \beta_3^{[P]}  
- \alpha_2^{[P]} \beta_5^{[P]}) }
{x_Q^2 + x_Q \alpha_2^{[P]} + \alpha_4^{[P]}}.
\end{align*}

After taking the limit of \eqref{DgPABdef} as  $P_2 \to \infty$ in a similar way, we obtain 
  formulas for the Mumford coordinates of the sum of 
two special divisors: $D_{1P} = P_1 = (x_P,y_P) $ and $D_{1Q} = Q_1=(x_Q,y_Q)$,
which, evidently, coincide with \eqref{AlphaBeta}:
\begin{align*}
&\alpha_2^{[P+Q]} =  - x_P - x_Q,&  
&\alpha_4^{[P+Q]} =  x_P x_Q,&\\
&\beta_3^{[P+Q]} = - \frac{y_P - y_Q}{x_P-x_Q},&
&\beta_5^{[P+Q]} = \frac{x_Q y_P - x_P y_Q}{x_P-x_Q}.&
\end{align*}

\subsection{Duplication law}
We are also interested in duplication: $D_4 = 2 D_2$, when 
$D_{2P} = D_{2Q} = D_2$.
Let $D_2 = (x_1,y_1) + (x_2,y_2)$, and $\alpha_2$, $\alpha_4$, $\beta_3$, $\beta_5$
be defined by \eqref{AlphaBeta}. In this case, the system
of equations for $\gamma_k$ has the form
\begin{gather}\label{SEqsDupl}
\textsf{S}_6 = 0,\quad \textsf{S}_4 = 0,\quad
 \rmd_{x_1,x_2}\textsf{S}_6 = 0, \quad \rmd_{x_1,x_2} \textsf{S}_4 = 0,
\end{gather}
where $\displaystyle \rmd_{x_1,x_2} = \frac{\rmd}{\rmd x_1} + \frac{\rmd}{\rmd x_2}$.  
We denote
\begin{gather}\label{ABDupl}
\begin{split}
&\alpha'_2 = \rmd_{x_1,x_2} \alpha_2 = -2,\qquad\qquad\quad
\alpha'_4 = \rmd_{x_1,x_2} \alpha_4 = x_1 +x_2,\\
&\beta'_3 = \rmd_{x_1,x_2} \beta_3 = - \frac{y'_1 - y'_2}{x_1-x_2},\qquad
\beta'_5 =  \rmd_{x_1,x_2} \beta_5  = \frac{y'_1 x_2 - y'_2 x_1}{x_1-x_2} + \frac{y _1 - y_2}{x_1-x_2},
\end{split}
\end{gather}
where $y'_i$, $i=1$, $2$, are computed by
$$
y_i' = \lim_{(x,y) \to (x_i,y_i)} \frac{-\partial_x f(x,y;\lambda)}{\partial_y f(x,y;\lambda)} 
= \frac{ \mathcal{P}'(x_i)}{2y_i}.
$$
By taking into account that $\alpha'_4 = - \alpha_2$, and 
$\beta'_5 = \frac{y'_1 x_2 - y'_2 x_1}{x_1-x_2} - \beta_3$,  the system of equations
 \eqref{SEqsDupl} is reduced to the following matrix form
\begin{gather}\label{GammaInABEqsDuplSmpl}
\begin{pmatrix}
1 & 0 & -\alpha_4 & -\beta_5 \\
0 & 1 & -\alpha_2 & -\beta_3 \\
0 & 1 & 0 & -\beta'_5 -\beta_3  \\
0 & 0 & 2 & -\beta'_3
\end{pmatrix}
\begin{pmatrix}
\gamma_6 \\ \gamma_4 \\ \gamma_2 \\ \gamma_1
\end{pmatrix}
+  \begin{pmatrix}
\alpha_2 \alpha_4 \\ \alpha_2^2 - \alpha_4 \\ 
- 3 \alpha_4  \\ 
- 3 \alpha_2
\end{pmatrix}
= 0.
\end{gather}
A solution of \eqref{GammaInABEqsDuplSmpl} is given by
\begin{subequations}\label{gammaD2Dupl}
\begin{align}
 \begin{pmatrix} \gamma_2 \\ \gamma_1 \end{pmatrix} 
& =  \frac{1}{2 \beta'_5 - \alpha_2 \beta'_3}
\begin{pmatrix}
3 \alpha_2 \beta'_5 -  \big(\alpha_2^2 + 2 \alpha_4 \big) \beta'_3 \\ 
\alpha_2^2  - 4\alpha_4 
\end{pmatrix} ,\\
 \begin{pmatrix} \gamma_6 \\ \gamma_4 \end{pmatrix} 
&=  \begin{pmatrix} \alpha_4 & \beta_5 \\
 \alpha_2 & \beta_3 \end{pmatrix} 
 \begin{pmatrix} \gamma_2 \\ \gamma_1 \end{pmatrix} -
  \begin{pmatrix} \alpha_2 \alpha_4\\
  \alpha_2^2 - \alpha_4  \end{pmatrix},
\end{align}
\end{subequations}
or in terms of coordinates
\begin{align}\label{gammaD2DuplXY}
 \begin{pmatrix} \gamma_6 \\ \gamma_4 \\ \gamma_2 \\ \gamma_1 \end{pmatrix} 
& =  \begin{pmatrix} -\tfrac{1}{2} x_1x_2 (x_1+x_2) \\ 3 x_1x_2 \\ -\tfrac{3}{2}(x_1+x_2) \\ 0 \end{pmatrix} 
+ \frac{(x_1-x_2)^2}{(y'_1 + y'_2)(x_1-x_2) - 2(y_1-y_2)} \times \\
&\phantom{mmmmmmmmmm} \times   \begin{pmatrix}
- \tfrac{1}{2} x_1 x_2 (y'_1 - y'_2) + x_2 y_1 - x_1 y_2 \\
 - (x_2 y'_1 - x_1 y'_2) \\
\tfrac{1}{2} (y'_1 - y'_2) \\ 
- (x_1-x_2)
\end{pmatrix}. \notag
\end{align}

Finally, the reduced divisor $\widetilde{D}_{2} \sim 2D_{2Q}$
is defined by
\begin{subequations}\label{DivSumAlphaBetaDupl}
\begin{align}
\begin{split}\label{DivSumAlphaDupl}
\alpha_2^{[2Q]} &= - 2 \alpha_2^{[Q]} + 2\gamma_2 - \gamma_1^2, \\
\alpha_4^{[2Q]} &= - 2 \alpha_4^{[Q]} + 3\big(\alpha_2^{[Q]}\big)^2  
- 2 \alpha_2^{[Q]} \big(2\gamma_2 - \gamma_1^2 \big) 
 + 2 \gamma_4 + \gamma_2^2  - \lambda_2 \gamma_1^2,
\end{split}\\
\begin{split}\label{DivSumBetaDupl}
\beta_3^{[2Q]} &= - \gamma_1^{-1} \big( 
  (\alpha_2^{[2Q]})^2 - \alpha_4^{[2Q]}  - \gamma_2 \alpha_2^{[2Q]} + \gamma_4\big),\\
\beta_5^{[2Q]} &=  - \gamma_1^{-1} \big(
 \alpha_2^{[2Q]} \alpha_4^{[2Q]} -  \gamma_2 \alpha_4^{[2Q]} + \gamma_6 \big),
\end{split}
\end{align}
\end{subequations}
where $\alpha_2^{[Q]} \equiv \alpha_2$, $\alpha_4^{[Q]} \equiv \alpha_4$, 
$\beta_3^{[Q]} \equiv \beta_3$, $\beta_5^{[Q]} \equiv \beta_5$.

\subsection{The case of $D_4 \sim D_1$ }
The case of $D_{2P} \,{+}\, D_{2Q} \,{\sim}\, D_1$, when the result of addition
is a special divisor $D_1 = (x_{P+Q},y_{P+Q})$, deserves special consideration.
In this case, addition is implemented through a polynomial function of weight $5$ of the form
$$\widetilde{\mathcal{R}}_5(x,y) =  y + \gamma_1 x^2 + \gamma_3 x + \gamma_5,$$
whose divisor of zeros is composed of $D_{2P}$,  $D_{2Q}$, and $D_1^\ast$ such that $D_1^\ast = -D_1$. 

With $A=P$, $Q$  we have the equality
\begin{equation*}
\widetilde{\mathcal{R}}_5(x,y) = \mathcal{R}_5^{[A]}(x,y) 
+ \gamma_1 \mathcal{R}_4^{[A]}(x)
\equiv \textsf{S}_3^{[A]} x + \textsf{S}_5^{[A]},
\end{equation*}
where
\begin{equation*}
\textsf{S}_3^{[A]} = \gamma_3 - \alpha^{[A]}_2 \gamma_1 -\beta^{[A]}_3,\qquad
\textsf{S}_5^{[A]} = \gamma_5 - \alpha^{[A]}_4 \gamma_1 -\beta^{[A]}_5.
\end{equation*}
The overdetermined system 
\begin{equation}\label{SEqs}
\textsf{S}_3^{[P]}=0,\qquad \textsf{S}_5^{[P]}=0,\qquad
\textsf{S}_3^{[Q]}=0,\qquad \textsf{S}_5^{[Q]}=0
\end{equation}
is consistent, and produces the condition
\begin{equation}
\gamma_1 = - \frac{\beta_3^{[P]}-\beta_3^{[Q]}}{\alpha_2^{[P]} - \alpha_2^{[Q]}} 
= - \frac{\beta_5^{[P]}-\beta_5^{[Q]}}{\alpha_4^{[P]} - \alpha_4^{[Q]}} ,
\end{equation}
which singles out pairs of divisors $D_{2P}$,  $D_{2Q}$ whose sum is equivalent to
a special divisor $D_1$. This condition causes vanishing the denominator in
\eqref{gammaD2}.
From the system \eqref{SEqs} we also find 
\begin{gather}
\gamma_3 = -\frac{\alpha_2^{[Q]}  \beta_3^{[P]} - \alpha_2^{[P]}  \beta_3^{[Q]}}{\alpha_2^{[P]} - \alpha_2^{[Q]}} ,\qquad
\gamma_5 = \beta_5^{[P]} - \alpha_4^{[P]}  \frac{\beta_3^{[P]}-\beta_3^{[Q]}}{\alpha_2^{[P]} - \alpha_2^{[Q]}}.
\end{gather}

Finally, with the help of the equality
$ y_{\widetilde{\mathcal{R}}_5} \,{=}\, {-}  \big( \gamma_1 x^2 + \gamma_3 x + \gamma_5\big)$
used in
\begin{equation*}
 f(x,y_{\widetilde{\mathcal{R}}_5}; \lambda) = 
\mathcal{R}^{[P]}_{4}(x) \mathcal{R}^{[Q]}_{4}(x) \big(x-x_{P+Q}\big),
\end{equation*}
we obtain
\begin{gather}
\begin{split}
&x_{P+Q} = \alpha_2^{[P]} + \alpha_2^{[Q]} + \gamma_1^2 - \lambda_2,\\
&y_{P+Q} = \gamma_1 x_{P+Q}^2 + \gamma_3 x_{P+Q}  + \gamma_5.
\end{split}
\end{gather}

\bigskip
In the case of duplication $2D_{2Q} \sim D_1$, 
which leads to a special divisor $D_1 = (x_{2Q},y_{2Q})$, 
we replace the system \eqref{SEqs} with
\begin{equation}\label{SEqsSpec}
\textsf{S}_3^{[Q]}=0,\qquad \textsf{S}_5^{[Q]}=0,\qquad
 \rmd_{x_1,x_2}\textsf{S}_3^{[Q]}=0,\qquad  \rmd_{x_1,x_2}\textsf{S}_5^{[Q]}=0,
\end{equation}
and obtain the condition
\begin{equation}
2 \beta'_5{}^{[Q]} = \alpha_2{}^{[Q]} \beta'_3{}^{[Q]} ,\qquad \text{or}\qquad
\frac{1}{2} (y'_1 + y'_2) = \frac{y_1-y_2}{x_1-x_2},
\end{equation}
which singles out such $D_{2Q}$ that  $2D_{2Q} \sim D_1$. 
Then we find expressions for $\gamma_k$:
\begin{gather}\label{GammaDuplSpec}
\gamma_1 = \tfrac{1}{2} \beta'_3{}^{[Q]},\quad
\gamma_3 = \tfrac{1}{2} \big( 2 \beta_3^{[Q]} + \alpha_2^{[Q]}  \beta'_3{}^{[Q]} \big),\quad
\gamma_5 = \tfrac{1}{2} \big( 2 \beta_5^{[Q]} + \alpha_4^{[Q]}  \beta'_3{}^{[Q]} \big).
\end{gather}

Finally, we find coordinates of the resulting divisor $D_1$:
\begin{gather}
\begin{split}
&x_{2Q} = 2 \alpha_2^{[Q]} + \gamma_1^2 - \lambda_2,\\
&y_{2Q} = \gamma_1 x_{2Q}^2 + \gamma_3 x_{2Q}  + \gamma_5.
\end{split}
\end{gather}

\section{Torsion divisors}
Let  $D \in \mathcal{C}^2$ be a reduced divisor, special or non-special. 
We say that $D$ is an \emph{$n$-torsion} divisor, if $n D \sim O$, where $O = 2\infty$
denotes the neutral divisor, and $D$ 
generates a cyclic group of order $n$:
$C_n = \langle O, D \rangle$. 
This definition implies the following criterion.
\begin{theo}\label{T:TorsionPs}
A  divisor $D \in \mathcal{C}^2$ is $n$-torsion when the following holds
\begin{enumerate}
\item[(i)]  $(k+1) D \sim - k D$, if $n = 2k+1$; or
\item[(ii)]  $k D \sim - k D$,  if $n=2k$.
\end{enumerate}
\end{theo}

\begin{rem}\label{R:nTorsDivs}
$n$-Torsion divisors are the Abel pre-images of points of order $n$ on $\Jac(\mathcal{C})$, 
that is, of $u[\varepsilon] \in \Jac(\mathcal{C})$ computed by \eqref{uChar}
from characteristics $[\varepsilon]$ of order $n$. of such $u[\varepsilon]$.
An example of computations of $2$-, $3$- and $4$-torsion divisors on a genus four curve
is implemented in Wolfram Mathematica 12, and can be found at
 \texttt{https://community.wolfram.com/groups/-/m/t/3314103}.
 This example, as well as computations
 on a genus two curve, shows that $n$-torsion divisors are non-special, except the case of $n=2$.
\end{rem}

Condition (i) in Theorem~\ref{T:TorsionPs}  implies 
\begin{gather}\label{2k1TorsCrit}
\alpha_2^{[(k\,{+}\,1)D]} =  \alpha_2^{[k D]},\qquad \ \ 
\alpha_4^{[(k\,{+}\,1)D]} =  \alpha_4^{[k D]}.
\end{gather}
Relations for $\beta$-coordinates in this case are trivial.

Condition (ii)   implies that $k D$ is $2$-torsion, that is
\begin{gather}\label{2kTorsCrit}
\beta_3^{[k D]} = - \beta_3^{[k D]}  =0,\qquad
\beta_5^{[k D]}  = -\beta_5^{[k D]}  =0,
\end{gather}
if $k D$ is non-special. If $k D$ is special, then $k D \sim (x_{k D},y_{k D})$,
and instead of \eqref{2kTorsCrit} we have
\begin{gather}\label{2kTorsCritSpec}
y_{k D} = -y_{k D}  =0. \tag{\ref{2kTorsCrit}'}
\end{gather}
Relations for $\alpha$-coordi\-na\-tes in case (ii) are trivial.
Equations \eqref{2k1TorsCrit} and \eqref{2kTorsCrit} written 
in terms of $\alpha_2^{[D]}$, $\alpha_4^{[D]}$, $\beta_3^{[D]}$, $\beta_5^{[D]}$
can be considered as \emph{division polynomials} in Mumford coordinates.

\subsection{$2$-Torsion divisors}
$2$-Torsion divisors are defined by the condition
$y_i = 0$, which implies that $x_i$ are $x$-coordinates of branch points.
That is, $2$-torsion divisors are Abel pre-images of half-periods on $\Jac(\mathcal{C})$.
Among fifteen $2$-torsion divisors, ten are composed of two finite branch points,
and correspond to even characteristics. The remaining five $2$-torsion divisors 
are composed of infinity and  one finite branch point,
and correspond to odd characteristics.

\subsection{$3$-Torsion divisors}
Let $D_{2Q} = (x_1,y_1)+(x_2,y_2)$ be a $3$-torsion divisor, which means that $2 D_{2Q} \sim - D_{2Q}$.
Note, that all  $3$-torsion divisors are non-special. 
From \eqref{2k1TorsCrit} we find
 the criterion
\begin{gather}\label{3TorsCrit}
\alpha_2^{[2Q]} =  \alpha_2^{[Q]},\qquad 
\alpha_4^{[2Q]} =  \alpha_4^{[Q]}.
\end{gather}
Applying the duplication law \eqref{DivSumAlphaBetaDupl}, we obtain
the equations which define $3$-torsion divisors in terms of Mumford coordinates
($\alpha_2^{[Q]} \equiv \alpha_2$,  $\alpha_4^{[Q]} \equiv \alpha_4$):
\begin{gather}\label{Alpha3Tors} 
\begin{split}
&3\alpha_2 = 2\gamma_2 - \gamma_1^2, \\
&3 \alpha_4 =  3 \alpha_2^2 - 2 \alpha_2 \big(2\gamma_2 - \gamma_1^2\big)
  + 2 \gamma_4 + \gamma_2^2  - \lambda_2 \gamma_1^2,
\end{split}
\end{gather}
where parameters $\gamma_k$ are defined by \eqref{gammaD2DuplXY}.

Therefore, all pairs of points $(x_1,y_1)$, $(x_2,y_2)$ which form $3$-torsion divisors 
can be obtained from \eqref{Alpha3Tors}, 
subject to $f(x_1,y_1;\lambda)=0$, $f(x_2,y_2;\lambda)=0$.
Substituting \eqref{gammaD2DuplXY} and \eqref{AlphaBeta}, we rewrite 
these equations in terms of coordinates:
\begin{align*}
&  (y'_1 - y'_2)\Big((y'_1 + y'_2)(x_1-x_2) - 2(y_1-y_2)\Big) - (x_1-x_2)^4 = 0,\\
& {-}\big( x_2 (y'_1)^2 - x_1 (y'_2)^2 \big) (x_1-x_2)  
+ 2 \big(x_1 y'_1 - x_2 y'_2 \big) (y_1-y_2) \\
&\phantom{mmmmmm} 
-3 (y_1-y_2)^2 - (2x_1+2x_2 + \lambda_3) (x_1-x_2)^4=0, \notag
\end{align*}
or
\begin{subequations}\label{Tors3Eqs}
\begin{align}
&\mathcal{Y}(x_1,y_1;x_2,y_2)
\equiv y_1 y_2 \big(\mathcal{P}'(x_1) \mathcal{P}(x_2) + \mathcal{P}'(x_2) \mathcal{P}(x_1)\big) \label{Tors3Eq1} \\
&\quad + \tfrac{1}{4} (x_1-x_2) \big(\mathcal{P}'(x_1)^2 \mathcal{P}(x_2) 
- \mathcal{P}'(x_2)^2 \mathcal{P}(x_1)\big) \notag \\
&\quad - \mathcal{P}(x_1)\mathcal{P}(x_2) \big(\mathcal{P}'(x_1) + \mathcal{P}'(x_2) + (x_1-x_2)^4 \big) = 0, \notag\\
&y_1 y_2 \big( 6  \mathcal{P}(x_1)  \mathcal{P}(x_2)
- x_1 \mathcal{P}'(x_1)  \mathcal{P}(x_2)  - x_2 \mathcal{P}'(x_2) \mathcal{P}(x_1) \big)  \label{Tors3Eq2} \\
&\quad - \tfrac{1}{4} (x_1-x_2) \big(x_2 \mathcal{P}'(x_1)^2 \mathcal{P}(x_2) 
- x_1 \mathcal{P}'(x_2)^2 \mathcal{P}(x_1)\big) \notag \\
&\quad + \mathcal{P}(x_1)\mathcal{P}(x_2) \big( x_1 \mathcal{P}'(x_1) - 3 \mathcal{P}(x_1)
+ x_2 \mathcal{P}'(x_2)  - 3 \mathcal{P}(x_2)  \notag\\
&\qquad  - (2 x_1+2 x_2 + \lambda_2) (x_1-x_2)^4 \big) = 0. \notag
\end{align}
\end{subequations}
Finally, we eliminate $y_1 y_2$ from these two equations,
and cancel the common factor $(x_1-x_2)^4$, namely
\begin{multline}\label{Tors3EqX}
\mathcal{X}(x_1,x_2) \equiv - \frac{1}{4}  \big(\mathcal{P}(x_2)^2 \mathcal{P}'(x_1) \mathcal{T}(x_1)^2 
+ \mathcal{P}(x_1)^2 \mathcal{P}'(x_2)  \mathcal{T}(x_2)^2 \big)  \\
+ \frac{1}{2} \frac{\mathcal{P}(x_2)^3  \mathcal{T}(x_1)^2 - \mathcal{P}(x_1)^3  \mathcal{T}(x_2)^2}{(x_1-x_2)} 
-\frac{1}{4} \bigg( \frac{\mathcal{P}(x_1)-\mathcal{P}(x_2)}{x_1-x_2} \bigg)^5  \\
+ \frac{3}{4} \frac{\mathcal{P}(x_2)^2  \mathcal{T}(x_1) + \mathcal{P}(x_1)^2  \mathcal{T}(x_2)}{(x_1-x_2)} 
\bigg( \frac{\mathcal{P}(x_1)-\mathcal{P}(x_2)}{x_1-x_2} \bigg)^2  \\
- \mathcal{P}(x_1)  \mathcal{P}(x_2) 
\frac{\mathcal{P}(x_2)  \mathcal{P}'(x_1) - \mathcal{P}(x_1)  \mathcal{P}'(x_2) }{x_1-x_2}
\bigg(\frac{\mathcal{T}(x_1)+\mathcal{T}(x_2) }{x_1-x_2}\bigg) \\
+ \mathcal{P}(x_1)  \mathcal{P}(x_2) \Big( {-} 6  \mathcal{P}(x_1)  \mathcal{P}(x_2)
+ x_1 \mathcal{P}(x_2)  \mathcal{P}'(x_1) + x_2  \mathcal{P}(x_1) \mathcal{P}'(x_2) \\
+ \big(\mathcal{P}(x_2)  \mathcal{P}'(x_1) + \mathcal{P}(x_1)  \mathcal{P}'(x_2) \big) (2x_1+2x_2 + \lambda_2)\Big) = 0,
\end{multline}
where 
\begin{gather*}
\begin{split}
\mathcal{T}(x_i) &= \frac{\mathcal{P}(x_1) - \mathcal{P}(x_3) - \mathcal{P}'(x_i)(x_1-x_3)}{(x_1-x_3)^2}\\
& = \frac{1}{x_1-x_2} \Big( x_1^4 + x_1^3 x_2 + x_1^2 x_2^2 + x_1 x_3^3 + x_2^4 - 5 x_i^4 \\
&\quad + \lambda_2 \big(x_1^3 + x_1^2 x_2 + x_1 x_2^2 + x_2^3 - 4 x_i^3\big) \\
&\quad + \lambda_4 \big(x_1^2 + x_1 x_2 + x_3^2 - 3 x_i^2\big)
+ \lambda_6 (x_1+x_2 - 2x_i ) \Big),
\end{split}
\end{gather*}
which is a polynomial, as follows from the series expansion of $\mathcal{P}$ about $x_i$.
By direct computations, one can verify that the following functions are polynomials
\begin{gather*}
\frac{\mathcal{T}(x_1)+\mathcal{T}(x_2) }{x_1-x_2},\quad 
\frac{\mathcal{P}(x_2)^2 \mathcal{T}(x_1) + \mathcal{P}(x_1)^2 \mathcal{T}(x_2) }{x_1-x_2},\quad
\frac{\mathcal{P}(x_2)^3 \mathcal{T}(x_1)^2 - \mathcal{P}(x_1)^3 \mathcal{T}(x_2)^2 }{x_1-x_2}.
\end{gather*}

The polynomial $\mathcal{X}$ in \eqref{Tors3EqX} has weight $40$, and
vanishes on $40$ pairs $\{x_1$, $x_2\}$.
In the system \eqref{Tors3Eqs} we  replace \eqref{Tors3Eq2} with  \eqref{Tors3EqX}, 

\begin{theo}
The polynomials $\mathcal{X}$ and $\mathcal{Y}$ defined by  \eqref{Tors3EqX} and \eqref{Tors3Eq1} 
on a curve $\mathcal{C}$ defined by~\eqref{C25Eq} single out the collection of  $3$-torsion divisors, 
which are Abel pre-images of $u[\varepsilon] \in \Jac(\mathcal{C})$ 
with characteristics $[\varepsilon]$ of order $3$.
\end{theo}

\subsection{$4$-Torsion divisors}
There exist $240$ characteristics of order 4 excluding half-integer characteristics.
Let $\mathfrak{E}$ denote the set of these $240$ characteristics.
Each characteristic of $\mathfrak{E}$ produces a $4$-torsion divisor $D_{2Q}=(x_1,y_1)+(x_2,y_2)$, 
which is non-special.
We split $\mathfrak{E}$ into two parts: 
\begin{itemize}
\item $\mathfrak{E}_{\text{non-spec}} = 
\{[\varepsilon] \in \mathfrak{E} \mid \mathcal{A}^{-1}(u[2\varepsilon]) \,{=}\, D_2,\, \deg D_2 \,{=}\, 2\}$ 
of cardinality  $160$, the divisor $D_2 \sim 2 D_{2Q}$ is characterized by its
 Mumford coordinates $\alpha_2^{[2Q]}$,  $\alpha_4^{[2Q]}$,
 $\beta_3^{[2Q]}$, $\beta_5^{[2Q]}$;

\item
$\mathfrak{E}_{\text{spec}} 
= \{[\varepsilon] \in \mathfrak{E} \mid \mathcal{A}^{-1}(u[2\varepsilon]) \,{=}\, D_1,\, \deg D_1 \,{=}\, 1\}$ of cardinality  $80$, 
the divisor $D_1 \sim 2 D_{2Q}$ is characterized by $x$-, $y$-coordinates: $D_1 = (x_{2Q},y_{2Q})$.
\end{itemize}

Divisors $D_{2Q}$ produced from characteristics of $\mathfrak{E}_{\text{non-spec}}$ 
satisfy the conditions
\begin{gather}\label{4TorsCrit}
\beta_3^{[2Q]} = 0,\qquad 
\beta_5^{[2Q]} = 0,
\end{gather}
obtained from \eqref{2kTorsCrit}. In the Mumford coordinates of 
$D_{2Q}$ the equalities \eqref{4TorsCrit} acquire the form
($\alpha_2^{[Q]} \equiv \alpha_2$,  $\alpha_4^{[Q]} \equiv \alpha_4$)
\begin{gather}\label{Beta4Tors} 
\begin{split}
&- 2 \alpha_4 -  \alpha_2^2 + \big(2 \alpha_2 - \gamma_2 + \gamma_1^2\big) \big(\gamma_2 - \gamma_1^2\big)
+  \gamma_1^2 \big(\gamma_2 - \lambda_2\big) + \gamma_4 = 0, \\
&\big({-} 2 \alpha_4 + (3 \alpha_2 - \gamma_2)(\alpha_2 - \gamma_2) 
+  \gamma_1^2  (2 \alpha_2 - \lambda_2) + 2 \gamma_4 \big) 
\big(2 \alpha_2 - \gamma_2 + \gamma_1^2\big) = \gamma_6,
\end{split}
\end{gather}
where $\gamma_k$ are defined by \eqref{gammaD2Dupl}.

Divisors $D_{2Q}$ produced from characteristics of $\mathfrak{E}_{\text{spec}}$ 
satisfy the condition $y_{2Q} \,{=}\, 0$, which follows from \eqref{2kTorsCritSpec},
and in terms of the Mumford coordinates of $D_{2Q}$ acquires the form
\begin{gather}\label{Beta4TorsSpec} 
\big(\gamma_1 (2 \alpha_2 + \gamma_1^2 - \lambda_2) + \gamma_3\big) 
\big(2 \alpha_2 + \gamma_1^2 - \lambda_2 \big) + \gamma_5 = 0,
 \tag{\ref{Beta4Tors}'}
\end{gather}
where $\gamma_k$ are defined by \eqref{GammaDuplSpec}.

In \texttt{https://community.wolfram.com/groups/-/m/t/3338527}
the reader may find an example of computing  $3$- and $4$-torsion divisors
on a genus two curve, along with derivation of the corresponding 
division polynomials in Mumford coordinates, and
in $x$-, $y$-coordinates.

\appendix

\section{Transformation to canonical curve}\label{A:CTrans}
There are several types of equations which define  a genus two curve.
\begin{enumerate}
\item[(I)] Let a genus two curve has the form
\begin{align}\label{C25EqYSft}
0 = f(x,y) &= -y^2 + y \mathcal{Q}(x) +  \mathcal{P}(x) \notag  \\
&= -y^2 + y \big(\nu_1 x^2 + \nu_3 x + \nu_5 \big) \\
&\quad\ + x^5+ \nu_2 x^4 + \nu_4 x^3  + \nu_6 x^2  + \nu_8 x + \nu_{10}.  \notag
\end{align}
By the map $y \mapsto y+\tfrac{1}{2} \mathcal{Q}(x)$ 
the curve \eqref{C25EqYSft}  transforms into the form  \eqref{C25Eq}, namely
\begin{align*}
0 = f(x,y) &= -y^2 +  \Delta(x)  \\
&= -y^2 + x^5+ \lambda_2 x^4 + \lambda_4 x^3  
+ \lambda_6 x^2  + \lambda_8 x + \lambda_{10},  \notag
\end{align*}
where $\Delta(x) = \mathcal{P}(x) + \tfrac{1}{4} \mathcal{Q}(x)^2$, and
\begin{align}
&& &\lambda_2 = \nu_2 + \tfrac{1}{4} \nu_1^2,&
&\lambda_8 = \nu_8 + \tfrac{1}{2} \nu_3 \nu_5,&  \notag \\
&& &\lambda_4 = \nu_4 + \tfrac{1}{2} \nu_1 \nu_3,&
&\lambda_{10} = \nu_{10} + \tfrac{1}{4} \nu_5^2,&\\
&& &\lambda_6 = \nu_6 + \tfrac{1}{2} \nu_1 \nu_5 + \tfrac{1}{4} \nu_3^2. \notag
\end{align}

\item[(II)] Let a genus two curve has the form
\begin{align}\label{C26Eq}
0 &= f(x,y;\lambda) = -y^2  + \bar{\mathcal{P}}(x) \\
&= -y^2  + a_0 x^6 + a_1 x^5+ a_2 x^4 + a_3 x^3  + a_4 x^2  + a_5 x + a_6. \notag
\end{align}
Let $\{e_i\}_{i=0}^5$ be roots of $\bar{\mathcal{P}}$, ordered
ascendingly in the real part, and then in the imaginary part.
Transformation to a curve in the form \eqref{C25Eq} is realized by moving
the smallest root $e_0$ to infinity, namely
\begin{gather}\label{C26toC25}
x \mapsto e_0 + \frac{\bar{\mathcal{P}}'(e_0)}
{x - \frac{1}{10} \bar{\mathcal{P}}''(e_0)},\qquad 
y \mapsto \frac{y \bar{\mathcal{P}}'(e_0)}
{\big(x - \frac{1}{10} \bar{\mathcal{P}}''(e_0) \big)^3}.
\end{gather}

If a canonical basis is chosen
as explained in \cite[\S3.4]{BerCalc24}, such a transformation of the curve does not change 
the correspondence between characteristics and divisors.

\bigskip
\item[(III)] Let a genus two curve has the form
\begin{align}\label{C26EqYSft}
0 = f(x,y) &= -y^2 + y \bar{\mathcal{Q}}(x) +  \bar{\mathcal{P}}(x) \notag  \\
&= -y^2 + y \big(\bar{b}_0 x^3 +  \bar{b}_1 x^2 + \bar{b}_2 x + \bar{b}_3 \big) \\
&\quad\ + \bar{a}_0 x^6 + \bar{a}_1 x^5+ \bar{a}_2 x^4 + \bar{a}_3 x^3  
+ \bar{a}_4 x^2  + \bar{a}_5 x + \bar{a}_6.  \notag
\end{align}
By the map $y \mapsto y+\tfrac{1}{2} \bar{\mathcal{Q}}(x)$ 
the curve \eqref{C26EqYSft}  transforms into   \eqref{C26Eq}
\begin{align*}
0 &= f(x,y;\lambda) = -y^2  + \bar{\Delta}(x) \\
&= -y^2  + a_0 x^6 + a_1 x^5+ a_2 x^4 + a_3 x^3  + a_4 x^2  + a_5 x + a_6, \notag
\end{align*}
and then by \eqref{C26toC25} to the form \eqref{C25Eq}.
\end{enumerate}

\section{Addition law on a curve with extra terms}
Addition law on a curve of the form \eqref{C25EqYSft}:
\begin{equation}
\tag{\ref{DivSumAlpha}'}
\begin{split}
\alpha_2^{[P+Q]} &= - \alpha_2^{[P]} - \alpha_2^{[Q]} + 2\gamma_2 - \gamma_1^2 +  \nu_1 \gamma_1, \\
\alpha_4^{[P+Q]} &= - \alpha_4^{[P]} - \alpha_4^{[Q]} + \big(\alpha_2^{[P]}\big)^2  + \alpha_2^{[P]} \alpha_2^{[Q]} 
+ \big(\alpha_2^{[Q]}\big)^2 \\
&\quad\ - \big(\alpha_2^{[P]} + \alpha_2^{[Q]}\big)\big(2\gamma_2 - \gamma_1^2 +  \nu_1 \gamma_1\big) \\
&\quad\  + 2 \gamma_4 + \gamma_2^2 + \big(\nu_1 \gamma_2  - \nu_2 \gamma_1 + \nu_3\big) \gamma_1.
\end{split}
\end{equation}

Duplication law:
\begin{equation}
\tag{\ref{DivSumAlphaDupl}'}
\begin{split}
\alpha_2^{[2Q]} &= - 2\alpha_2^{[Q]}  + 2\gamma_2 - \gamma_1^2 +  \nu_1 \gamma_1, \\
\alpha_4^{[2Q]}  &= -2 \alpha_4^{[Q]}  + 3 (\alpha_2^{[Q]} )^2 - 2 \alpha_2^{[Q]}  
 \big(2\gamma_2 - \gamma_1^2 +  \nu_1 \gamma_1\big)\\
&\quad\  + 2 \gamma_4 + \gamma_2^2 
+ \big(\nu_1 \gamma_2  - \nu_2 \gamma_1 + \nu_3\big) \gamma_1,
\end{split}
\end{equation}


\end{document}